\documentclass[12pt]{amsart}

\usepackage{amssymb}
\usepackage{graphicx}
\usepackage{amscd, color}
\usepackage{amsmath}
\usepackage{amsfonts}
\usepackage[english]{babel}
\usepackage{epsfig}
\usepackage{xypic}
\usepackage{color}



\thanks{This work was supported by the
 Ministerio de Econom\'{\i}a y Competitividad (Spain) under grants MTM2011-22417 (P. Rueda) and 
 MTM2012-36740-C02-02 (E.A. S\'anchez P\'erez).}

\newtheorem{theorem}{Theorem}
\newtheorem{proposition}[theorem]{Proposition}
\newtheorem{corollary}[theorem]{Corollary}
\newtheorem{example}[theorem]{Example}

\newtheorem{remark}[theorem]{Remark}

\newtheorem{lemma}[theorem]{Lemma}
\newtheorem{definition}[theorem]{Definition}

\title[The Dvoretsky-Rogers Th. for vector valued integrals]{ The Dvoretsky-Rogers Theorem for  vector valued  integrals on function spaces }
\author{P. Rueda and E.A. S\'anchez P\'erez}

\begin{document}

\maketitle

\begin{abstract}
 We show a  Dvoretsky-Rogers type  Theorem for the adapted version of the $q$-summing operators to the topology of the convergence of the vector valued integrals  on Banach function spaces. In the pursuit of this objective we prove that  the mere summability of the identity map does not guaranty that the space has to be finite dimensional, contrarily to the classical case. Some local compactness assumptions on the unit balls are required. Our results open the door to new convergence theorems and tools regarding summability of series of integrable functions and approximation in function spaces, since we may find infinite dimensional spaces in which convergence of the integrals ---our vector valued version of convergence in the weak topology--- is equivalent to the convergence with respect to the norm.    Examples and applications are also given.
\end{abstract}

\section{Introduction}

Summability  in Banach spaces is one of the main topics in
applied analysis, and results regarding the behavior of summable sequences are fundamental tool for its applications. Comparison between norm and weak absolutely summable series is at the origin of some classical problems in the theory of Banach spaces, and it was the starting point of the theory of $p$-summing operators. In this paper we are interested in providing new elements for the analysis of summability in the case of Banach function spaces by using a vector valued duality, that is provided by the vector measure integration theory on spaces $L^p(m)$ of integrable functions   with respect to a vector measure $m$.  These spaces represent, in fact,
all order continuous $p$-convex Banach lattices with weak unit. %according to the structure of these spaces.
This theory supplies a distinguished element
---the vector valued integral--- for the study of summability in Banach spaces of measurable functions. It is well known that $fg\in L^1(m)$ whenever $f\in L^p(m)$ and $g\in L^{p'}(m)$, $1/p+1/p'=1$. In this case, the integral $\int fg \, dm$ determines a  vector valued bilinear map that yields to a duality: the  vector valued duality  between $L^p(m)$ and
$L^{p'}(m)$ (see \cite{illi,proc}).

 This vector valued duality is the  framework to study  natural topologies on spaces of integrable functions with respect to a vector measure, as the topology $\tau_m$ generated by the seminorms $\gamma_g(f):=\|\int fg\, dm\|$, $f\in L^p(m)$, when varying $g\in L^{p'}(m)$. This new vector valued point of view was first  taken into consideration in the study of convergence of sequences: the relation between the convergence of sequences in  spaces of vector measure integrable functions and the convergence of the corresponding
 vector valued integrals has been treated since the seventies (see for instance \cite{lewis1,lewis2}, \cite[Section 6]{cur}, \cite{libro} and the references therein).
%In \cite{lewis1,lewis2} ---at the beginning of the seventies---
%Lewis already noticed that the pointwise convergence of the integrals plays a
%fundamental role in this setting, that is different than the role
%played by the norm convergence and the weak convergence. The first example of this different was given by
% Curbera beginning of the nineties, went further in this direction
%
%
In this paper we are interested in the summability of sequences  in $L^p(m)$ spaces induced by the vector valued duality, that is, when the role played by the weak topology is assumed by the topology $\tau_m$. It is worth mentioning that the
$p$-convexification $L^p(m)$ ($p \ge 1$) of the space $L^1(m)$ of
a vector measure $m$ was introduced as a tool for analyzing
summability (see \cite{illi}), trying to bring together vector
valued integration and the theory of $p$-summing operators in
Banach spaces (see also \cite{Fe,irene}).

The classical
Dvoretsky-Rogers Theorem can be stated as follows: the identity map in a Banach space $E$ is absolutely $q$-summing for some $1\leq q<\infty$, if and only if $E$ is finite dimensional. This paper is devoted to prove an extension for Banach function spaces of this result. In our context, the usual scalar duality is replaced by the vector valued duality given by a vector measure and the role of the weak topology in the Banach space is assumed by  the topology $\tau_m$.  In order to develop our study, we analyze some properties of the $(q,P^m)$-summing operators, that map $\tau_m$
summable sequences to norm summable sequences.  Our main result shows the necessity of adding some topological requirements on local compactness to characterize finite dimensional spaces in terms of the   $(q,P^m)$-summability of the identity map. The last section shows an application to the study of subspaces of $L^p(m)$ that are fixed by the integration operator. As a consequence of our Dvorestky-Rogers type theorem, we prove that, under the local  compactness hypotheses, only finite dimensional subspaces can be  fixed by the integration map.

\section{Preliminaries}

We use standard Banach space notation. Let $1 \le p \le \infty$. Then we write $p'$
for the extended real number satisfying $1/p+1/p'=1$. We follow the
definition of  Banach function space over a finite measure $\mu$
given in \cite[Def.1.b.17, p.28]{lint}. Throughout the paper $X(\mu)$ will denote an infinite dimensional Banach function space over $\mu$, i.e. $X(\mu)$ is a Banach lattice of $\mu-a.e.$-equal classes of $\mu$-integrable functions with a lattice norm and the $\mu$-a.e. order satisfying $L^\infty(\mu) \subseteq X(\mu) \subseteq L^1(\mu)$. We will also assume that $X(\mu)$ is order continuous, that is, for each decreasing sequence $f_n \downarrow 0$ in $X(\mu)$, $\lim_n \|f_n\|_{X(\mu)}=0.$

Let $X$ be a real Banach space and let $(\Omega, \Sigma)$ be a
measurable space. If
$m:\Sigma\to X$ is a  countably additive vector measure, we write $\mathcal{R}(m)$ for its range. The variation $|m|$ of $m$ is given by $|m|(A):=\sup_{B_i \in \pi}
\sum_{i=1}^n \|m(B_i)\|$, where the supremum is computed over all
finite measurable partitions $\pi$ of $A \in \Sigma$.
$\|m\|$ is the semivariation of $m$, i.e.$\|m\|(A):=\sup_{x^* \in B_{X^*}} | \langle m, x^* \rangle|(A)$,
$A \in \Sigma$, where $\langle m,x^* \rangle$ is the scalar
measure given by $\langle m,x^* \rangle(A):= \langle m(A),x^*
\rangle$. The Rybakov Theorem (see \cite[Ch. IX]{dies}) establishes that
  there exists $x^*\in
X^*$ such that $m$ is absolutely continuous with respect to a so called Rybakov measure
$|\langle m,x^*\rangle|$, that means that  $m(A) = 0$ whenever $|\langle
m(A),x^*\rangle|= 0$. For $1 \le p <
\infty$, a (real) measurable function $f$ is said to be
$p$-integrable with respect to $m$ if $|f|^p$ is integrable with
respect to all measures $|\langle m, x^* \rangle|$ and for each $A
\in \Sigma$ there exists an element $\int_A |f|^p  dm \in X$ such
that $\langle \int_A |f|^p dm, x^* \rangle =\int_A |f|^p d \langle
m, x^* \rangle$, $x^* \in X^*$.

The space $L^p(m)$, $1 \le p < \infty$, is defined to be the Banach lattice of all
($\mu$-equivalence classes of) measurable real functions defined
on $\Omega$ that are $p$-integrable with respect to $m$ when the
a.e. order and the norm
$$
\|f\|_{L^p(m)}:= \Big( \sup_{x^* \in B_{X^*}} \int |f|^p d
|\langle m, x^* \rangle| \Big)^{1/p}, \quad f \in L^p(m),
$$
are considered. It is an order continuous $p$-convex
Banach function space over any Rybakov measure $\eta$ for $m$
(see \cite[Proposition 5]{illi}; see also
\cite{fernandez-mayoral-naranjo-saez-sanchez perez} and
\cite[Ch.3]{libro} for more information on these spaces).
For the case $p=\infty$, $L^\infty(m)$ is defined as $L^\infty(\eta)$.
A
relevant fact is that for each $1 \le p \le \infty$, $L^p(m)
\cdot L^{p'}(m) \subseteq L^1(m)$ (see \cite[Prop.3.43]{libro} and
\cite[Sec.3]{illi}; see also
\cite{fernandez-mayoral-naranjo-saez-sanchez perez}). Moreover,
for each $f \in L^p(m)$
\begin{equation} \label{foru}
\|f\|_{L^p(m)}= \sup_{g \in B_{L^{p'}(m)}} \Big\| \int fg \, dm\Big\|.
\end{equation}
These relations allows to define the so called vector measure duality by using the
  integration operator
$I_m:L^1(m)\to X$, which is given by
$$
I_m(f)=\int_\Omega f \, dm, \ f\in L^1(m).
$$
We will use the symbol $\int f \, dm$ instead of $\int_\Omega f \, dm$ throughout the paper.
Relevant information on the properties of $I_m$ can be found in
\cite{okr2,ORiRoPi,orrp1}, and \cite[Ch.3]{libro} and the references
therein. Since for all $p>1$ the inclusion $L^p(m) \subseteq
L^1(m)$ always holds, the integration map can be defined also as
an operator $I_m:L^p(m) \to X$; we use the same symbol $I_m$ in this case for this operator.  It must be said that the
spaces $L^p(m)$ represent in fact the class of all order
continuous $p$-convex Banach lattices with a weak unit (see
\cite[Prop.2.4]{fernandez-mayoral-naranjo-saez-sanchez perez} or
\cite[Prop.3.30]{libro}), what means that our results can be
applied to a broad class of Banach spaces.

As we said in the Introduction, duality and vector valued duality for the spaces $L^p(m)$ are fundamental tools in this paper. Regarding duality, fix a Rybakov measure $\mu$ for $m$. Due to the order continuity
of $L^p(m)$, its dual space $L^p(m)^*$ ($1 \le p < \infty$) allows
an easy description; it coincides with its K\"{o}the dual (or
associate space) $(L^p(m))'$, that is,
$L^p(m)^*=(L^p(m))'=\{\varphi_g:g \in \mathcal H\}$, where
$$
\mathcal H:=\{g:\Omega\to \mathbb R \ \  \Sigma-\mbox{measurable
}:fg\in L^1(\mu) \mbox{ for all }f\in L^p(m)\}
$$
and the duality is given by
  $\langle \varphi_{g},f\rangle=\int_\Omega fg\, d\mu.$ Information about a precise description of $(L^p(m))'$
  can be found in \cite{ferrsa,Fe,FeRo,galaz,proc}. It must be said here that $(L^p(m))'$ and $L^{p'}(m)$ coincide only in very special situations, for instance for $m$ being a scalar measure. We will write  $\tau_w$ for the weak topology  on $L^p(m)$.

Regarding vector valued duality relations between $L^p(m)$ spaces, $1 \le p \le \infty$, the integration map defines the continuous bilinear map
$$
B_m:L^p(m) \times L^{p'}(m) \to X
$$
given by $B_m(f,g):= \int f g \, dm $, $f \in L^p(m)$, $ g\in L^{p'}(m)$.
Note that $B_m$ is both sides norming for $L^p(m)$ and $L^{p'}(m)$, that is, for every $f \in L^p(m)$,
$\|f\|_{L^p(m)}= \sup_{g \in B_{L^{p'}(m)}} \| \int fg \, dm \|$, and the same happens dually for the case of functions $g \in L^{p'}(m)$.

 In this paper we will consider the topology $\tau_m$ of  pointwise convergence of the
integrals, i.e. the locally convex topology  defined by the
seminorms $\gamma_g(f):= \|\int fg dm \|_X$, $f \in L^p(m)$, $g \in
L^{p'}(m)$.
The topology  $\tau_{w,m}$  of  pointwise weak convergence of the
integrals, is defined by the seminorms
$\gamma_{g,x^*}(f):= \langle \int fg dm, x^* \rangle$, $f \in
L^p(m)$, $g \in L^{p'}(m)$, $x^* \in X^*$. It  is also  a locally convex topology on $L^p(m)$.
It is easy to see that the norm topology is finer than all the others, and
$\tau_m$ and $\tau_{w}$ are finer than $\tau_{w,m}$, although $\tau_m$ and $\tau_{w}$ are not comparable in  general.  An exhaustive analysis of the $\tau_m$ topology has been done recently and can be found in \cite{RuSPTMNA} (see also the references therein).  The reader can find more information about it in \cite{fernandez-mayoral-naranjo-saez-sanchez
perez,Fe,FeRo,libro, RuSPpreprint,illi}. The following result establishes the basic relations between the quoted topologies.

\begin{proposition} \label{compact} (Proposition 1 in \cite{RuSPTMNA})

 Let $1\leq p \le \infty$.
 If $B_{L^p(m)}$ is $\tau_{m}$-compact then $\tau_{w,m}$ and $\tau_{m}$ coincide on bounded subsets of $L^p(m)$.
 Moreover, if $p>1$ and $B_{L^p(m)}$ is $\tau_{m}$-compact then the weak topology and $\tau_m$ coincide on bounded subsets of $L^p(m)$.
 Consequently, if $p>1$, $B_{L^p(m)}$ is $\tau_{m}$-compact if and only if $(L^p(m),\|\cdot\|_{L^p(m)})$ is
  reflexive and the weak topology and $\tau_{m}$ coincide on $B_{L^p(m)}$.
\end{proposition}

In this paper we will make a local use of the duality defined by the integration bilinear map $B_m$. For $1 \le p \le \infty$ consider a subspace $P \subseteq L^p(m)$. We say that a subspace $R \subseteq L^{p'}(m)$ is an {\it $m$-dual} for $P$ if $R$ is $m$-norming for $P$, i.e. the function $f \rightsquigarrow \sup_{g \in B_R} \| \int fg \,dm\|$ gives an \textbf{equivalent} norm for $P$. We write $P^m$ for such a space $R$. In the same way, we say that a subspace $P^{m m}$ of $L^p(m)$ is an $m$-bidual of $P$ (with respect to an $m$-dual $P^m$) if $P \subseteq P^{m m}$ and $P^{m m}$ is $m$-norming for $P^m$. Notice that the inclusion $P \subseteq P^{mm}$ is not necessary for $P^{mm}$ to be $m$-norming for $P^m$. For instance, if $X(\mu)$ is an order continuous Banach function space and $m:\Sigma \to X(\mu)$ is the vector measure given by $m(A):= \chi_A,$ $A \in \Sigma$, then for $P=L^p(m)$ the space $P^m$ generated by the function $\chi_\Omega$ in $L^{p'}(m)$ is $m$-norming for $P$, and also the space $P^{mm}$ generated by $\chi_\Omega$ in $L^p(m)$ is $m$-norming for $P^m$. However, $P$ is not included in $P^{mm}$. But note also that given $P$, $P^m$ and $P^{mm}$ being norming, it can always be assumed that $P \subseteq P^{mm}$ just by defining the new $P^{mm}$ as the subspace of $L^p(m)$ generated by $P \cup P^{mm}.$
We will use this example later.

 We say that a triple $(P,P^m,P^{m m})$  of $m$-dual spaces as above is an $m$-dual system. We can define the topology $\tau_m(P^m)$ over $P$ as the one induced by all the seminorms $f \rightsquigarrow \| \int  f g \, dm \|$, $g \in P^m$, and $\tau(P^{m m})$ the topology for $P^m$ given by the seminorms $g \rightsquigarrow \| \int  f g \, dm \|$, $f \in P^{m m}$. A quick look at  the proof of Proposition 1 in \cite{RuSPTMNA} shows that a local version of this result is also true, i.e. a version of this result writing $P$ instead of $L^p(m)$ and $\tau_m(P^m)$ instead of $\tau_m$, where $P^m$ is an $m$-dual space.

Let us show some examples. A natural $m$-dual space of  $L^p(m)$ is $L^{p'}(m)$; in this case, we write simply $\tau_m$ for the topology $\tau_m(L^{p'}(m))$.  However, an $m$-dual space may be very small. For instance, if the integration map $I_m:L^1(m) \to X$ is an isomorphism, then the subspace generated by $\chi_\Omega \in L^\infty(m)$ is an $m$-dual for $L^1(m)$. Obviously, for every subspace $P \subseteq L^p(m)$, $L^{p'}(m)$ is an $m$-dual for $P$.

%\section{Norm summability and summability of the integrals: $(q,R)$-summing operators}

Let us finish  this section by defining a fundamental class of operators related to the summability of
sequences with respect to the $\tau_m$-topology. It
generalizes the class considered in Lemma 16 of \cite{illi} and in
\cite[Section 4.2]{Fe}. Theorem 17 in \cite{illi} provides
 a Pietsch type domination/factorization theorem for this family of operators. The local version of this result becomes the main tool for the proof of our results.

\begin{definition}
Let $1 \le p,q < \infty$, $R$ be a subspace of $L^{p'}(m)$ and $P$ a Banach subspace of $L^p(m)$.
Let $E$ be a Banach space. An
operator $T: P \to E$ is $(q,R)$-summing if there is a
constant $K$ such that for any finite set of functions
$f_1,...,f_n \in P$,
$$
\Big( \sum_{i=1}^n \|T(f_i)\|^q \Big)^{1/q} \le K \sup_{g \in
B_R} \Big( \sum_{i=1}^n \Big\| \int f_i g \, dm \Big\|^q \Big)^{1/q}.
$$
\end{definition}

Of course, the integration map $I_m:P \to X$ is always
$(q,L^{p'}(m))$-summing for all $1 \le p,q < \infty$. Indeed, if $f_1,...,f_n \in L^p(m)$, then
$$
\sum_{i=1}^n \| I_m(f_i) \|^q \le \|m\|(\Omega)^{q/p} \cdot
\sup_{h \in B_{L^{p'}(m)}} \Big( \sum_{i=1}^n \| \int f_i h
dm\|^q \Big).
$$

\section{The Dvoretsky-Rogers Theorem for the $m$-summability}

Throughout  this section, $1 \le p \le \infty$, $E$ and $X$ are Banach spaces, $m$ is an $X$-valued vector measure, $P$ is a subspace of $L^p(m)$ and $(P,P^m,P^{m m})$ is an $m$-dual system. We will consider the following sequential properties associated to compactness with respect to the $\tau_m$-topology.

\begin{definition}
 An operator $T:E \to P$ is $\tau_m(P^m)$-sequentially compact if every bounded sequence $(x_n)_n$ in $E$ has a subsequence $(x_{n_k})_k$
such that   $( \int  T(x_{n_k}) g dm )_k$ is a Cauchy sequence for each $g \in
P^m$.
\end{definition}

\begin{definition}
 An operator $T:P \to E$ is
$\tau_m(P^m)$-sequentially completely continuous if $\lim_n \|T(f_n)\|=0$  whenever $(f_n)_n$ is a bounded sequence
such that   $\lim_n \| \int h f_n dm \|=0$ for every $h \in
P^m$.
\end{definition}

If we assume that $\chi_\Omega\in P^m$ (we can always make $P^m$ big enough to have it), then $I_m:P\to X$ is $\tau_m(P^m)$-sequentially continuous. In the classical summing operators theory it is well known that any summing operator is weakly compact. However, not every $(q,P^m)$-summing operator is $\tau_m(P^m)$-sequentially compact. For instance, given a Banach function space $X(\mu)$, define $m(A):=\chi_A$, $A\in \Sigma$. Then $I_m:L^1(m)\to X(\mu)$ is an isomorphism which is $(q,L^\infty(m))$-summing but it is not $\tau_m(L^\infty(m))$-sequentially compact in general as in this case the norm topology and the $\tau_m=\tau_m(L^\infty(m))$ topology coincide. Then $B_{L^1(m)}$ is not $\tau_m$-compact unless $L^1(m)$ is finite dimensional. Let us see that under some compactness assumptions, the $(q,P^m)$-summing operators behave similarly as absolutely summing operators. We need first an easy lemma.

\begin{lemma} \label{inclu}
Let $1 \le p \le \infty$ and let $P$ be a subspace of $L^p(m)$ and $P^m$ an $m$-norming subspace for $P$. Consider a Banach space valued
$(r,P^m)$-summing operator $T:P \to E$. Then $T$ is $(q,P^m)$-summing for each $1 \le r\le q < \infty$.
\end{lemma}
\begin{proof} Let $s$ be such that $\textstyle \frac 1q+\frac 1s=\frac 1r$.
Take a finite set of functions $f_1,...,f_n \in P$. Then
\begin{eqnarray*}
\big( \sum_{i=1}^n \|T(f_i)\|^q \big)^{1/q} &=& \sup_{(\lambda_i)_{i=1}^n \in B_{\ell^{s}}}\left( \sum_{i=1}^n | \lambda_i |^r\| T(f_i)\|^r\right)^{1/r}\\
& \le &\sup_{(\lambda_i)_{i=1}^n \in B_{\ell^{s}}} \left(\sum_{i=1}^n  \| T(\lambda_i f_i)\|^r\right)^{1/r}\\
& \le &
 \sup_{(\lambda_i)_{i=1}^n \in B_{\ell^{s}}} K \sup_{g \in B_{P^m}} \left(\sum_{i=1}^n  |\lambda_i |^r \Big\|\int f_i g dm\Big\|^r\right)^{1/r} \\
& \le & \sup_{g \in B_{P^m}} \left( \sum_{i=1}^n \Big\|\int f_i g dm \Big\|^q \right)^{1/q}\\
\end{eqnarray*}
 where $K$ is the constant associated to the $(1,P^m)$-summability of $T$.
\end{proof}

\begin{theorem} \label{lemafacto}
Let $1 \le q < \infty$.  Let $T:P \to E$ be a
$(q,P^m)$-summing operator. The following statements hold.
\begin{itemize}
\item[(i)] If
$B_{P^m}$ is $\tau_m(P^{m m})$-compact then $T$ is $\tau_m(P^m)$-sequentially completely continuous.
\item[(ii)] If $B_{P}$ is $\tau_m(P^m)$-compact and $B_{P^m}$ is $\tau_m(P^{m m})$-compact then $T$ is completely continuous.

\item[(iii)]
Finally, if $B_{P^m}$ is $\tau_m(P^{m m
})$-compact  and  $X$ is reflexive,
 then $T$ is also weakly compact.
\end{itemize}
\end{theorem}
\begin{proof}
(i) We have that $T$ satisfies that for every finite set $f_1,...,f_n \in
P$,
$$
\sum_{i=1}^n \| T(f_i ) \|^q \le K^q \sup_{g \in B_{P^m}}
\sum_{i=1}^n \Big\| \int  f_i g dm \Big\|^q.
$$
Taking into account that $(P,P^m,P^{m m})$ is an $m$-dual system,
it can be shown as in the case of  Pietsch's Domination Theorem for $q$-summing operators
 ---see Lemma 17 in \cite{illi} and make the obvious modifications--- that there is
measure $\eta$ on the compact space $(B_{P^m}, \tau_m(P^{m m}))$ such
that
$$
\|T(f)\| \le K \left( \int_{B_{P^m}} \Big\| \int f g \, dm \Big\|^q \,
d \eta(g)  \, \right)^{1/q}, \quad f \in P.
$$
 This easily gives that $T$ factorizes
through the following scheme (see Theorem 17 in \cite{illi})
$$
\xymatrix@R=3pc@C=3pc{
P \ar[r]^{T} \ar[d]_j & E \\
C_0 \ar[r]^{i} & S_0, \ar[u]_u }
$$
where $C_0$ is the subspace of $C(B_{P^m},X)$ given by the
functions $g \rightsquigarrow \int f g dm \in X$, $j$ the
isomorphism given by the identification of a function $f$ with  the corresponding vector valued function in $C_0$,
$S_0$ is the closure of the image of $C_0$ by the natural
inclusion/quotient map
$$
C(B_{P^m},X) \to
L^q(B_{P^m},\eta,X),
$$
where $\eta$ is a Radon probability measure on
$B_{P^m}$, and $u$ is the map that closes the diagram. Using
this scheme, an argument based on the Dominated Convergence
Theorem gives the result. Let $(h_n)_n$ be a bounded sequence in $P$ such that the sequence of integrals $(\| \int  h_n g
dm\|)_n$ is null for every $g \in P^m$. It is enough to prove that the
 sequence of functions $g \rightsquigarrow \int  h_n g
dm \in X$  satisfies
$\lim_n \int_{g \in B_{P^m}} \| \int h_n g dm\|^qd \eta(g)
=0$. For each $n$, the function $\varphi_n(\cdot):=\| \int
h_n \, \cdot \, dm\|$ belongs to the space $C(B_{P^m})$ of scalar continuous functions  defined on  the compact set
$(B_{P^m},\tau_m(P^{m m}))$. Since there is a constant $K>0$ such that
$\varphi_n(g) \le K \chi_{B_{P^m}}(g)$ for all $g \in
B_{P^m}$ and $n$, we can apply the Lebesgue Dominated Convergence Theorem to obtain that
 $$
 \lim_n \int_{B_{P^m}} \Big\| \int h_n g dm\Big\|^q d \eta(g) = \lim_n \int_{B_{P^m}} \varphi^q_n(g) d \eta(g)
  =0.
 $$
Therefore, using the factorization we obtain that $\lim_n\|T(h_n)\| =0$ and so $T$ is $\tau_m(P^m)$-sequentially completely continuous.

(ii)  Let $(f_n)_{n}$ be a  weakly null
sequence  in $P$. Since $(B_{P^m},\tau_m(P^{m m}))$ is  compact, using the factorization given in (i) and taking into account that each continuous operator is weak-to-weak continuous, we get that for each element
$\delta_g \otimes x' \in (C(B_{P^{m}},X))'$, $g \in
P^m$, $x' \in X'$, we have that
$$
\lim_n \langle \int f_n \,\cdot \,  dm, \delta_g \otimes x' \rangle = \lim_n \langle \int  f_n g dm, x' \rangle =0.
$$
Due to an easy adaptation of Proposition \ref{compact}, since we are assuming that
$(B_P, \tau_m(P^m))$ is compact, the topologies
$\tau_{w,m}(P^m)$, generated by the seminorms $\gamma_{x^*,g}(f):=\langle \int fg\, dm,x^*\rangle$ when varying $x^*\in X'$ and $g\in P^m$,  and $\tau_m(P^m)$ coincide on $B_{P
}$.

Consequently for each $g \in P^m$,
 $\lim_n \| \int f_n g dm \| =0.$ Using the domination in (i), we obtain the result on the complete continuity.

 (iii) Finally, by Lemma \ref{inclu} if $T$ is
 $(q,P^m)$-summing it is  $(s,P^m)$-summing for $q<s<\infty$, and so the reflexivity of $X$ implies the reflexivity of $L^s(B_{P^m},\eta,X)$. Thus,
the factorization of $T$ through a subspace of
$L^s(B_{P^m},\eta,X)$ gives that $T$ is weakly compact.

\end{proof}

The following result is a direct consequence of the statements (ii) and (iii) of Theorem \ref{lemafacto}.

\begin{corollary} \label{coroflex}
Suppose that $m$ is an $X$-valued vector measure and $X$ is reflexive.
Let $T:P \to P$ be a $(q,P^m)$-summing operator, and suppose that $B_{P}$ is $\tau_m(P^{m})$-compact and $B_{P^m}$ is $\tau_m(P^{m m})$-compact. Then $T \circ T$ is compact.
\end{corollary}

%\begin{corollary} \label{inclu} ???????????
%Let $P \subseteq L^p(m)$ and $R \subseteq L^{p'}(m)$ be a pair of $m$-dual subspaces and let $1 \le s \le q < \infty$. Suppose that $B_{R}$ is $\tau_m(P)$-compact.
%Then each operator $T:P \to E$ that is $(s, R)$-summing is also $(q,R)$-summing.
%\end{corollary}
%\begin{proof}
%It is a consequence of the domination given in the previous theorem, since for each vector valued integrable function $\phi$,
%$ \big( \int \|\phi\|^s \, d\eta \big)^{1/s} \le \big( \int \|\phi\|^q \, d\eta \big)^{1/q}$.
%\end{proof}
%
%\begin{remark}  �??????????????????
%If $B_{L^{p'}(m)}$ is $\tau_m$-compact, then the integration map $I_m:L^p(m) \to E$ is $(1,{L^{p'}(m)})$-summing and so $(q,{L^{p'}(m)})$-summing for every $q>1$. This is a consequence of Corollary \ref{inclu}, but a direct argument shows that the result is also valid without the assumption of compactness for  $B_{L^{p'}(m)}$.  It is known that for $p>1$, the
%integration operator restricted to $L^p(m)$ is weakly compact (see
%\cite[Cor.3.4]{fernandez-mayoral-naranjo-saez-sanchez perez} or
%\cite[Ch. 3]{libro}).
%Regarding compactness, it is also known that the range of the vector measure is relatively
%compact if and only if the integration map restricted to $L^p(m)$
%is compact for some $p>1$ (see
%\cite[Th. 3.6]{fernandez-mayoral-naranjo-saez-sanchez perez} or
%\cite[Prop.3.56]{libro}).
%\end{remark}

In particular, if $T:P \to E$ is an isomorphism in Corollary \ref{coroflex},  we obtain that $P$ has to be finite dimensional.

\begin{example}\rm
{\it A proper infinite dimensional subspace of a space $L^2(m)$ with an $m$-dual system in which $P$, $P^m$ and $P^{m m}$ coincide, but the identity map is not $(q,P^m)$-summing for any $1 \le q < \infty$.} Take an infinite non-trivial measurable partition $\{A_i\}_{i=1}^\infty$ of the Lebesgue space $([0,1], \mathcal B,\mu)$, and define the vector measure $m:\mathcal B \to \ell^2$ given by $m(B):=\sum_{i=1}^\infty \mu(A_i \cap B) e_i$, where $\{e_i:i=1,... \}$ is the canonical basis of $\ell^2$ and $B \in \mathcal B$ (see Example 10 in \cite{SP}). Consider the (infinite dimensional closed) subspace $P$ of $L^2(m)$ generated by the functions  $\chi_{A_i}/\mu(A_i)^{1/2}$, $ i \in \mathbb N$. A direct calculation shows that for each $f=\sum_{i=1}^\infty \lambda_i \chi_{A_i}/\mu(A_i)^{1/2} \in P$,
$$
\|f\|_{L^2(m)}= \left( \sum_{i=1}^\infty |\lambda_i|^4 \right)^{1/4},
$$
and so $P$ is isometric to $\ell^4$ (see Proposition 11 in \cite{SP}). We can define the $m$-dual space $P^m \subseteq L^2(m)$ and the $m$-bidual space $P^{m m}$ as $P^m=P^{m m}=P$. It is clear that $P^m$ norms $P$ and $P^{m m}$ norms $P^m$. However, the identity map is not $(q,P^m)$-summing for any $1 \le q < \infty$. In order to see this, consider the sequence of functions
$(\frac{\chi_{A_i}}{\mu(A_i)^{1/2}})_{i=1}^\infty$. Then, if $1 \le q < \infty$, for each $k \in \mathbb N$ we get
$$
\sum_{i=1}^k \Big\|\frac{\chi_{A_i}}{\mu(A_i)^{1/2}}\Big\|^q_{L^2(m)} = k,
$$
but
$$
\sup_{g \in B_{P^m}} \sum_{i=1}^k \Big\| \int \frac{\chi_{A_i}}{\mu(A_i)^{1/2}} g \, dm \Big\|^q_{\ell^2}=
\sup_{(\tau_i)_{i=1}^\infty \in B_{\ell^4}} \sum_{i=1}^k |\tau_i|^{q}
$$
$$
\le \sup_{(\tau_i)_{i=1}^\infty \in B_{\ell^4}}  \Big( \sum_{i=1}^k |\tau_i|^{4q} \Big)^{1/4} \cdot
k^{3/4} \le k^{3/4}.
$$
This gives a contradiction and shows that the identity map cannot be $(q,P^m)$-summing for any $1 \le q < \infty$.
Note that the range of $m$ is relatively compact, since it can be included in the convex hull of a null sequence of $\ell^2$.
Corollary 8 in \cite{RuSPTMNA} establishes that for a reflexive and separable space $L^2(m)$ ---our space satisfy both requirements---, relative compactness of the range of $m$ implies
compactness of $(B_{L^2(m)},\tau_m)$. $B_P$ is $\tau_m$-closed, since by Proposition \ref{compact},  $\tau_m$ is finer than the weak topology on $L^2(m)$. This gives compactness of $(B_P,\tau_m(P^m))$ ---since the topology $\tau_m(P^m)$ is weaker than the topology $\tau_m$ on $B_P$--- and so compactness of $(B_{P^m},\tau_m(P^{m m}))$. The topological requirements of Corollary~\ref{coroflex} are then satisfied and $P$ is reflexive, but obviously the identity map is not compact. Since $\ell^4$ is not a Schur space, the identity map is not completely continuous. This shows that the summability condition in Theorem~\ref{lemafacto} (ii) and in Corollary~\ref{coroflex} cannot be dropped.
\end{example}

The following is our main result and gives a vector measure version of the Dvoretsky-Rogers Theorem.

\begin{theorem} \label{DR}
Let $E$ be a Banach space, $P$ be a subspace of $L^p(m)$  and $T:P \to E$ be an isomorphism.
The following statements are equivalent.
\begin{itemize}
\item[(i)] There is an $m$-dual system $(P,P^m,P^{m m})$ such that $B_P$ is $\tau_m(P^m)$-sequentially compact, $B_{P^m}$ is $\tau_m(P^{m m})$-compact and $T$ is $(q,P^m)$-summing for some ---and then for all--- $1 \le q < \infty$.
\item[(ii)] $P$ has finite dimension.
\end{itemize}
\end{theorem}
\begin{proof}

(i) $\Rightarrow$ (ii) Assume that $T$ is $(q,P^m)$-summing for a fixed $1 \le q < \infty$. Let us show that the composition $T \circ T^{-1}$ is compact. As a consequence of Theorem \ref{lemafacto}(i), we know that $T$ is $\tau_m(P^m)$-sequentially completely continuous. Since
$B_{P}$ is $\tau_m(P^m)$-sequentially compact, $T^{-1}:T(E) \to P$ is $\tau_m(P^m)$-sequentially compact. %Moreover,  each norm bounded $\tau_m(P^m)$-Cauchy sequence in $P$ has a limit in $P$, and so the composition with a $\tau_m(P^m)$-sequentially completely continuous operator is compact.
Then the identity map $T \circ T^{-1}:P\to P$ is compact, and so $P$ has finite dimension.

(ii) $\Rightarrow$ (i) Since $P$ is finite dimensional, we have that $(S_P, \| \cdot\|_{L^p(m)})$ is compact. The norm topology is finer than $\tau_m$, and so the unit sphere $(S_P,\tau_m)$ is compact too. For each element $f \in S_P$, take a norm one function $g_f \in L^{p'}(m)$ that satisfies that $1/2 \le \| \int f g_f \, dm \| \le 1$. %Take $1/2 > \varepsilon >0$ and
Consider the $\tau_m$-open covering of $S_P$ given by the sets
$$
\Big\{ h \in L^p(m): \Big\|\int (h-f) g_f d m \Big\| <\frac 14, \, f \in S_P \Big\}.
$$
There is a finite subcovering given by a finite set $\mathcal C= \{ g_{f_i}: i=1,...,n \}$ of such functions $g_f$. Then we define $P^m$ to be the subspace generated by $\mathcal C.$
Note that for each $f \in S_P$ there is an index $i \in \{1,...,n \}$ such that $\|\int (f_i-f) g_{f_i} dm \| <\frac 14$ and so
\begin{eqnarray*}
\frac{1}{2} &\le&  \Big\| \int f_i g_{f_i} dm \Big\| \\
&\le& \Big\|\int f g_{f_i} dm \Big\| + \Big\|\int (f_i-f) g_{f_i} dm \Big\|\\
&\le&
\Big\|\int f g_{f_i} dm \Big\| + \frac 14 \\
& \le &\sup_{g \in B_{P^m}} \Big\| \int f g \, dm \Big\|  + \frac 14 \le \|f\|_{L^p(m)} \cdot \sup_{g \in B_{P^m}} \|g\|_{L^{p'}(m)}
+\frac 14\\
& \le& 1 +\frac 14.\\
\end{eqnarray*}
Consequently, for each $f \in P$,
\begin{equation}\label{111}
\frac 14 \|f\|_{L^p(m)} \le \sup_{g \in B_{P^m}} \Big\| \int f g \, dm \Big\|  \le  \|f\|_{L^p(m)}.
\end{equation}
Therefore, the space $P^m$ is $m$-norming for $P$, and $B_P$ is $\tau_m(P^m)$-sequentially compact since the norm topology and $\tau_m(P^m)$ coincides in the finite dimensional space $P$.

 Note that we can also define a finite dimensional subspace $P^{m m}$ containing $P$ that is $m$-norming for $P^m$ following the same procedure that in the definition of $P^m$. The finite dimension of $P^{m}$ proves also that $B_{P^m}$ is $\tau_m(P^{m m})$-compact.

 Finally, let us see that $T$ is $(q,P^m)$-summing for all $1 \le q$. By Lemma \ref{inclu} it suffices to prove that $T$ is $(1,P^m)$-summing. Write now $P'$ for the (usual topological) dual of $P$.  Since  $P$ is finite dimensional, we have that the identity map is $1$-summing, and so for each finite family $h_1,...,h_l \in P$
$$
\sum_{i=1}^l \|T(h_i)\| \le \|T\| \sum_{i=1}^l \|h_i\|_{L^p(m)} \le \|T\| K \sup_{y' \in B_{P'}} \sum_{i=1}^l |\langle h_i, y' \rangle|
$$
$$
= \|T\| K  \sup_{\epsilon_i =\pm 1 } \Big\| \sum_{i=1}^l \epsilon_i h_i \Big\|_{L^p(m)} \le
4 \|T\| K \sup_{\epsilon_i=\pm 1  } \sup_{g \in B_{P^m}}  \Big\|\int \sum_{i=1}^l \epsilon_i h_i g dm \Big\|
$$
$$
\le
4 \|T\| K \sup_{g \in B_R} \sum_{i=1}^n \Big\|\int h_i g dm\Big\|,
$$
where $K$ is the $1$-summing norm of the identity map  and the constant 4 comes from \eqref{111}. Therefore, $T$ is $(1,P^m)$-summing and so $(q,P^m)$-summing for every $ q\geq 1$.
\end{proof}

When $m$ is a scalar measure then the spaces $L^p(m)$ and $L^{p'}(m)$, $1<p<\infty$, are reflexive and hence their closed unit  balls are weakly compact or, equivalently, $\tau_m$-compact. Besides, in this case $(q,L^{p'}(m))$-summability coincides with the usual absolute $q$-summability for operators. Therefore  Theorem~\ref{DR} can be considered an extension of the classical Dvorestky-Rogers Theorem to  spaces of integrable functions with respect to a vector measure.

Let us present some examples that show that all the requirements in (i) are needed for the result to be true. Recall that $X(\mu)$ is an order continuous Banach function space over a finite measure space $(\Omega,\Sigma,\mu)$.

\vspace{.5cm}

\begin{remark} \label{contra}\rm

1. \textit{$\tau_m(P^m)$-sequential compactness of $B_P$ is a necessary requirement.} Consider the vector measure $m:\Sigma \to X(\mu)$ given by $m(A):=\chi_A$, $A \in \Sigma$. In this case, $L^1(m)=X(\mu)$ and the integration map $I_m:L^1(m)\to X(\mu)$ is an isomorphism. Take $P=L^1(m)$, that is not finite dimensional by assumption. The subspace $P^m$ of $L^{\infty}(m)$ generated by $\chi_\Omega$ is $m$-norming for $P$. Consider the $m$-bidual space $P^{m m}$ for $P$ defined as $P^{m m}= L^{1}(m)$. Obviously, $B_{P^m}$ is $\tau_m(P^{m m})$-compact. Since the seminorm on $L^1(m)$ defined by  $f \rightsquigarrow \| \int f \chi_\Omega dm \| = \|f\|$ coincides with the norm, we have that $ P^{m}$ is $m$-norming for $P$ but clearly $B_P$ is not $\tau_m(P^m)$-sequentially compact. Note that any other $m$-dual space for $P$ containing a function $g(w) > \delta$ for some $\delta >0$ satisfy the same property: $B_P$ is not compact for the topology $\tau_m(P^m)$. Observe also that the identity  $L^1(m)\to L^1(m)$ is $(q,P^m)$-summing for each $1 \le q < \infty$, since for each finite set $f_1,...,f_m \in L^1(m)$,
$$
\sum_{i=1}^n \| f_i \|^q_{L^1(m)} = \sum_{i=1}^n \Big\| \int f_i \chi_\Omega dm \Big\|^q_{X(\mu)}.
$$
Note that the identity is  $\tau_m(P^m)$ sequentially completely continuous trivially.
This example shows clearly the difference between $q$-summing and $(q,P^m)$-summing operators. In the first case, Alaoglu's Theorem assures that the unit ball of the dual space is weak*-compact, and this is enough to prove the Dvoretsky-Rogers Theorem via Pietsch's Factorization Theorem. In the second case, the topological properties for the unit balls of the spaces involved must be given as additional requirements. This means that the corresponding summability property for the isomorphism  does not assure our Dvoretsky-Rogers type theorem to hold.

\vspace{.5cm}

%(2) \textit{The inclusion $P \subseteq P^{m m}$.} Consider the same vector measure given above and define the spaces $P=L^1(m)$, $P^m= span \{\chi_\Omega \} \subseteq L^{\infty}(m)$ and
%$P^{m m}= span\{\chi_\Omega \} \subseteq L^1(m)$. Recall that $L^1(m)$ has not finite dimension. Both spaces $P^m$ and $P^{m m}$  are $m$-norming for $P$ and $P^m$, respectively, but the condition $P \subseteq P^{m m}$ fails. Regarding the compactness requirements,  $B_{L^1(m)}$ is not $\tau_m(P^m)$-sequentially compact ---$\tau_m(P^m)$ is the norm topology---, but $B_{P^m}$ is $\tau_m(P^{m m})$-compact.  Also, we have shown in (1) that the identity map is $(q,P^m)$-summing for every $1 \le q < \infty$.
%
%\vspace{.5cm}

2. \textit{Not all the $m$-dual systems for a finite dimensional space $P$ satisfy the requirements of Theorem \ref{DR}.}
Consider again the vector measure given in Example 1. Take $P$ as the (finite dimensional) subspace of $L^1(m)$ generated by $\chi_\Omega$.
First,
 take the $m$-dual system $P=P^m=P^{m m}$, with the understanding that $P$ and $P^{mm}$ are subspaces of $L^1(m)$ and $P^{m}$ is a subspace of $L^\infty(m)$. In this case,
  $B_P$ is $\tau_m(P^m)$-sequentially compact, $B_{P^m}$ is $\tau_m(P^{m m})$-compact and the identity map on $P$ ---that coincides with the integration operator--- is $(q,P^m)$-summable for each $1 \le q < \infty$, providing all the requirements in (i) of Corollary \ref{DR}.

However,
 take now $P^m=L^\infty(m)$ and $P^{m m}=L^1(m)$. Assume that the vector measure $m$ has not relatively compact range (for example, when  $X(\mu)=L^r[0,1]$, $1 \le r < \infty$, see Example 3.61 in \cite{libro}). Then $B_P$ is $\tau_m(P^m)$-sequentially compact but $B_{P^m}=B_{L^\infty(m)}$ is not $\tau_m(P^{m m})$-compact, since the topology $\tau_m$ induced on $L^\infty(m)=L^\infty(\mu)$ by $L^1(m)$ coincides with the topology of $X(\mu)$ on  this space. To see this, just consider the seminorm
 $$
 {L^\infty(m)} \ni g \rightsquigarrow \Big\| \int \chi_\Omega g \, dm \Big\| =  \|g\|_{X(\mu)}.
 $$
 Thus if $B_{P^m}$ is $\tau_m(P^{m m})$-compact, this would imply compactness of $B_{L^\infty(m)}$ with respect to the topology of $X(\mu)$, and so it would imply that the range of the vector measure is relatively compact, since it is included in $B_{L^\infty(m)}$.

\vspace{.5cm}

3. \textit{The topological requirements for the $m$-dual system are not enough: the assumption on the $(q,P_m)$-summability of the isomorphism is also needed.}  Consider the vector measure $m$ defined as Lebesgue measure $\mu$ on $[0,1]$. Take any $1 < p < \infty$ and consider $P=L^p(\mu)$. Then we have that  $P^m=L^{p'}(\mu)$ is an $m$-dual for $L^{p}(\mu)$, and so the topology $\tau_m(L^{p'}(\mu))$ gives the weak topology for the reflexive space $L^p(m)$
(see Proposition \ref{compact}) . If we define the $m$-bidual $P^{m m}=L^p(m)$, we have that the topology $\tau_m(P^{m m})$ for $P^{m}$ is given by the weak topology for $L^{p'}[0,1]$. So both topological requirements in (i) of Corollary \ref{DR} are satisfied.  Of course, no isomorphism from $P$  is  $q$-summing for any $1 \le q < \infty$, and so no isomorphism  is $(q,P^m)$-summing, since in this case both definitions of summability coincides.

\end{remark}

\vspace{.5cm}

\begin{example}\rm
\textit{ The vector measure associated to the Volterra operator.} Let $1 \le r < \infty$ and let $\nu_r:\mathcal{B}([0,1]) \to L^r([0,1])$ be the Volterra measure, i.e. the vector measure associated to the Volterra operator. This measure is defined as
 $$
 \nu_r(A)(t):= \int_0^t \chi_A(u) \, du \in L^r([0,1]), \quad A \in \Sigma
 $$
 (see the explanation in \cite[p.113]{libro}; all the information about this measure can be found in different sections of \cite{libro}).  It is known that the range of $\nu_r$ is relatively compact. This is a consequence of the compactness of the Volterra operator (see the comments after \cite[Proposition 3.47]{libro}).

 Let $1 <p<\infty$, $m=\nu_r$ and consider a subspace $P$ of $L^p(m)=L^p(\nu_r)$. Assume that there is an $m$-dual space $P^m$ for $P$ such that $B_P \subseteq K B_{L^\infty(m)}$ for a certain $K>0$ (for example a subspace generated by a finite set of functions in $L^\infty(m)$ with $L^p(m)$-norm greater than $\delta >0$). Take $P^{m m}$ as $L^p(m)$. Then $B_{P^m}$ is  $\tau_m(P^{m m})$-compact as a consequence of Theorem 10 in \cite{RuSPTMNA}.
 In this case, we have a simplified version of our Dvoretsky-Rogers type theorem for the subspace $P$: $P$ is finite dimensional if there is $1 \le q < \infty$ such that the identity map is $(q,P^m)$-summing and $B_P$ is $\tau_m(P^m)$-sequentially compact.
 %$L^p(\nu_r)$ is weakly sequentially complete (it is a consequence of Corollary 3.40 in \cite{libro}). DEL RESTO DE CONDICIONES NI IDEA!!!
%
%
% \begin{corollary} \label{corcompap}
%Let $1<p<\infty$ and let $m$ be a vector measure such that
% $L^p(m)$ is
%weakly sequentially complete ---equivalently, $L^p(m)$ is reflexive---, and
%$(B_{L^{p'}(m)}, \|.\|_{L^{p'}(m)})$ is separable. If $\mathcal{R}(m)$ is relatively
%compact then $(B_{L^p(m)},\tau_m)$ is compact. For $p=1$ the
%result is true under the assumption that $I_m:L^1(m) \to E$ is
%compact.
%\end{corollary}

\end{example}

\section{An application: subspaces of $L^p(m)$ that are fixed by the integration map}

 In what follows we use our results in order to obtain information about subspaces of $L^p(m)$ spaces that are fixed by the integration map $I_m$. This topic has been studied since the very beginning of the investigations on the structure of the spaces of integrable functions with respect to a vector measure, and several papers on this topic  have been published recently (mainly regarding subspaces that are isomorphic to $c_0$ and $\ell^1$, see \cite{ORSPDiss} and the references therein).
Let us show  an easy example.

\begin{example} \label{exfacil}\rm
Consider as in Remark \ref{contra} for $X(\mu)=L^r[0,1]$ the vector measure $m: \Sigma \to L^r[0,1]$ given by $m(A):=
\chi_A$, $r \ge 1$.
Consider the subspace $S$ generated by the Rademacher sequence in
$L^r[0,1]$. By the Kinchine inequalities, $S$ is a subspace in $L^r[0,1]$ that is isomorphic to $\ell^2$. Recall that $L^1(m)=L^r[0,1]$ and the integration map is an isomorphism. Obviously the restriction of the integral operator $I_m:L^1(m)\to L^r[0,1]$ to $S$ is in fact the identity map. For $p \ge 1$ we have that
$L^p(m)=L^{pr}[0,1]$, and again by the Kinchine inequalities $S$ is a subspace of $L^p(m)$ that is fixed by the integration map $I_m:L^p(m) \to L^r[0,1]$.
\end{example}
 As we noted after the definition of $(q,L^{p'}(m))$-summing operator, the integration map from $L^p(m)$ for any $1 \le p < \infty$ is always  $(q,L^{p'}(m))$-summing for every $q \ge 1$; in fact it is in a sense the canonical example of this kind of operators. Thus, our Dvoretsky-Rogers type result can be directly applied to obtain negative results on the existence of infinite dimensional subspaces of $L^p(m)$ that are fixed by $I_m$. We say that a subspace $P$ of $L^p(m)$ is {\it fixed by the integration map} if $I_m|_P$ is an
isomorphism into.

The following result  shows that under some compactness requirements,
any  subspace $S$ of $L^p(m)$ that is fixed by $I_m$ has to be finite dimensional.
For the case when the $m$-dual system that is considered is $P^m=L^{p'}(m)$ and $P^{m m}= L^p(m)$, conditions under which the balls of these spaces are $\tau_m$ compact are given in Corollary 8  of \cite{RuSPTMNA}.

%\begin{corollary} \label{DRth0}
%Let $p \ge 1$. Suppose that $B_{L^{p'}(m)}$ is $\tau_m$-compact.
%Consider a reflexive subspace $S$
%of $L^p(m)$ such that $I_m|_S$ is completely continuous. Then $S$ has finite dimension.
%\end{corollary}
%The proof is a direct consequence of  Remark \ref{opsub}, Corollary \ref{suco} and Proposition \ref{nece}. MIRARRRRRRRR LA PRUEBAAAAAA

%To finish, let us write a corollary of Theorem \ref{DR}.

\begin{corollary}\label{DRth}
Let $1 \le p < \infty$, and let $P$ be a subspace of $L^p(m)$ that is fixed by the integration map. If
there is an $m$-dual system for $P$ such that $B_{P}$ is $\tau_m(P^m)$-sequentially compact and $B_{P^m}$ is $\tau_m(P^{m m})$-compact, then $P$
 is finite dimensional.
\end{corollary}
\begin{proof}
It is a consequence of Theorem \ref{DR} and the fact that the integration map is $(q, L^{p'}(m))$-summing for every $1\leq q\leq \infty$.
\end{proof}

In particular, the subspace $P$ generated by  the Rademacher functions that has been shown in Example \ref{exfacil} does not have an $m$-dual system satisfying the compactness requirements in Proposition \ref{DRth}.
%Note also in the example just using the definition of $(q,S^m)$-summing operator that if we take the $m$-norming subspace $S^m:=L^{p'}(m)=L^{p' r}[0,1]$ for $S$, every Banach space valued operator is $(q,S^m)$-summable for every $ 1 \le q < \infty$.
%Let us show now that this fact is actually a general property: if a subspace is fixed by the integration map, the identity map in it is always $(q,P^m)$-summing for all $1 \le q < \infty$.

%\begin{corollary}\label{noco}
%\textcolor{red}{Let $1 \le p < \infty$.
%If $P$ is an infinite dimensional  subspace of $L^p(m)$ that is fixed by the integration map, then either $(B_P,\tau_m)$ is not sequentially compact or $(B_{L^{p'}(m)},\tau_m)$ is not compact.}
%\end{corollary}

\begin{remark}\rm
By \cite[Theorem 3.6]{fernandez-mayoral-naranjo-saez-sanchez perez},  if the vector measure $m$ has relatively compact range and $1 <p<\infty$, then the restriction of the integration map  to $L^p(m)$ is compact. Thus, if $S$ is a subspace of $L^p(m)$ that is fixed by the integration map, it has always finite dimension.
\end{remark}

%\begin{theorem}
%\textcolor{red}{Let $1 < p < \infty$, let $X$ be a weakly sequentially complete Banach space and let $m$ be an $X$-valued vector measure with relatively compact range. If  $L^{p}(m)$ is separable, then no infinite dimensional  subspace of $L^p(m)$, is fixed by the integration map.}
%\end{theorem}
%\begin{proof}
%Since $X$ is weakly sequentially complete, by \cite[Corollary 3.10]{fernandez-mayoral-naranjo-saez-sanchez perez} $L^p(m)$ and $L^{p'}(m)$ are reflexive. Separability of $L^p(m)$ implies separability of $L^{p'}(m)$ too by \cite[Proposition 2.3]{fernandez-mayoral-naranjo-saez-sanchez perez}, and so $B_P$ and $B_{L^{p'}(m)}$ are norm separable.
%Adapting the proof of \cite[Corollary~8]{RuSPTMNA} $(B_P,\tau_m)$ and $(B_{L^{p'}(m)},\tau_m)$ are compact. Since $L^{p'}(m)$ is separable then $L^p(m)$ is metrizable and so $(B_P,\tau_m)$ is sequentially compact. Proposition \ref{DRth} gives now the result.
%\end{proof}
%
%\begin{proposition}
%Let $1 < p < \infty$ and let $m$ be a vector measure with relatively compact range such that  $(B_{L^{p'}(m)},\tau_m)$ is  complete. Then no infinite dimensional  subspace $P$ of $L^p(m)$ such that $(B_P,\tau_m)$ is sequentially compact, is fixed by the integration map.
%\end{proposition}
%\begin{proof}
%Adapting the proof of  \cite[Theorem~7]{RuSPTMNA} to $P$ instead of $L^p(m)$, we get that $(B_{L^{p'}(m)},\tau_m)$ is compact. Proposition \ref{DRth} gives now the result.
%\end{proof}

To finish, let us  remark that as a consequence of the following result, the ideas that prove Corollary \ref{DRth} can be applied to maps acting in a subspace $P$ that is fixed by the integration map, others than the inclusion map.

\begin{proposition} \label{nece}
Let $1 \le p < \infty$.
Let $P$ be a subspace of $L^p(m)$ that is fixed by the integration map and let $P^m \subseteq L^{p'}(m)$ be an $m$-dual space of $P$ containing $\chi_\Omega$. Then
every operator $T:P \to F$ with values on a Banach space $F$ is $(q,P^m)$-summable for every $1 \le q < \infty$.
\end{proposition}
\begin{proof}
Let $T:P\to F$ be an operator with values on a Banach space $F$, and let $f_1,...,f_n \in P$. Then
$$
  \sum_{i=1}^n \| T(f_i) \|^q \le  \|T\|^q \cdot  \sum_{i=1}^n \| f_i
  \|_{L^p(m)}^q
   \le \|T\|^q  \cdot \|(I_m)^{-1}\|^q \cdot \sum_{i=1}^n \Big\| \int f_n dm
 \Big\|^q
 $$
 $$
 \le  \|T\|^q \cdot \|(I_m)^{-1}\|^q \cdot \|m\|(\Omega)^{q/p} \cdot \sup_{g \in B_{P^m}} \sum_{i=1}^n \Big\| \int f_n g dm\Big\|^q.
 $$
 This gives the result.
\end{proof}

%However, this is not the case for all the vector measures. Let us show a relevant example of subspace that is fixed by the integration map.

%\begin{remark}
%Notice that reflexivity of $S$ is not needed if $p \ne 1$ and the
%space $E$ where the vector measure is defined is reflexive, since
%in this case the Bochner space $L^q(B_{L^\infty(m)},E)$ is
%reflexive and so the operator $T$ becomes also weakly compact.
%Composing $T$ with its inverse and then with itself we get that
%$T$ is compact, again a contradiction.
%\end{remark}

%\begin{corollary} ESTO NO SIRVE PARA NADA SOLO CON EL REMARK DE ANTES SIN MAS CONDICIONES DA EL RESULTADO CONM CON RANGO COMPACTO SOLO
%Let $1<p  < \infty$ and suppose that $\mathcal{R}(m)$ is relatively
%compact, and $L^p(m)$ is separable and weakly sequentially complete.
%If  $S$ is a subspace
%of $L^p(m)$ that is fixed by the integration map, then $S$ has finite dimension.
%\end{corollary}
%\begin{proof}
%Corollary 8 in \cite{RuSPpreprint}.
%This a consequence of Corollary
%\ref{corcompap}, Theorem 3.6 in \cite{fernandez-mayoral-naranjo-saez-sanchez perez} and Corollary \ref{DRth}, just taking into account that separability and weak sequential completeness of $L^p(m)$ holds if and only if they hold for $L^{p'}(m)$.
%\end{proof}

\noindent[Pilar Rueda] Departamento de An\'alisis Matem\'atico, Universidad de
Valencia, 46100 Burjassot - Valencia, Spain, e-mail: pilar.rueda@uv.es

\medskip

\noindent[Enrique A. S\'anchez P\'erez] Instituto Universitario de Matem\'{a}tica Pura y Aplicada, Universitat Polit\`ecnica de Val\`encia, Camino de Vera s/n, 46022 Valencia, Spain, e-mail: easancpe@mat.upv.es

\end{document}